\documentclass{amsart}

\usepackage{amsmath,amsthm,amssymb,microtype,color,verbatim,hyperref,multicol}
\usepackage{amsmath,amsthm,amssymb,microtype,color,verbatim,hyperref,multicol,tikz}
\usepackage{comment}
\makeatletter
\def\imod#1{\allowbreak\mkern10mu({\operator@font mod}\,\,#1)}
\makeatother

\usepackage{mathtools}

\newtheorem{theorem}{Theorem}[section]
\newtheorem{lemma}[theorem]{Lemma}
\newtheorem{proposition}[theorem]{Proposition}

\newtheorem{conjecture}[theorem]{Conjecture}

\theoremstyle{definition}
\newtheorem{definition}[theorem]{Definition}
\theoremstyle{remark}
\newtheorem{remark}[theorem]{Remark}

\title[Avoiding monochromatic sub-hyperpaths]{Avoiding monochromatic sub-paths in uniform hypergraph paths and cycles\footnote{The authors wish to acknowledge support from the Center for Undergraduate Research in Mathematics and the National Science Foundation through NSF grant DMS-1722563.}
}

\author[W.~Billings]{W. Zane Billings}

\author[J.~Clifton]{Justin Clifton}

\author[J.~Hiller]{Josh~Hiller }

\author[T.~Meek]{Tommy Meek}

\author[A.~Penland]{Andrew Penland}

\author[W.~Rogers]{Wesley Rogers}

\author[G.~Smokovich]{Gabriella~Smokovich }

\author[A.~Velasquez-Berroteran]{Andrew~Velasquez-Berroteran }

\author[E.~Zamagias]{Eleni~Zamagias }

\date{\today}

\begin{document}

\maketitle

\markboth{Z. BILLINGS ET AL}{AVOIDING MONOCHROMATIC SUBPATHS}

\begin{abstract}
      We present a recursive formula for the number of ways to color $j$ vertices blue in an r-uniform hyperpath of size $n$ while avoiding a blue monochromatic sub-hyperpath of length k. We use this result to solve the corresponding problem for $r-1$-tight r-uniform paths and loose r-uniform cycles. This generalizes some well known results from reliability engineering and analysis.      
\end{abstract}

\section{Introduction}

Let $G(n)$ denote the number of binary strings of length $n$ which do not contain two consecutive 1's. It is well-known, but still pleasing to notice, that $G(n)$ satisfies the recurrence
\[
G(n+2) = G(n+1) + G(n)
\]
and is equal to the Fibonacci numbers, with indices shifted by 2. 

More generally, the count $G_c(n)$ of binary strings of length $n$ with no length-$c$ run of 1's satisfies the recurrence
\[
G_c(n+c) = \displaystyle\sum_{i=0}^{c-1} G_c(n+i)
\]
and has a well-known asymptotic apporoximation due to Feller~\cite{F1968}. 

If we add the additional parameter of the number of 1's in the string, the problem becomes more involved. A real-world motivation for this restriction comes from \textit{consecutive k-out-of-n:F systems} that arise in reliability engineering. Such a system consists of a line of $n$ identical components, and continues to function properly as long as no $k$ consecutive components have failed. A fundamental question, first approached in~\cite{S1981}, is given that exactly $j$ components have failed, what is the probability that the system still functions?

We can translate this question to graph theory. We view the components of the system as the vertices of a path, with functioning vertices colored red and failed vertices colored blue. Solving the original probability question leads to counting how many red/blue-colored paths of length $n$ have exactly $j$ blue vertices, but do not contain a monochromatic blue path of length $k$. If we let $F(n,k,j)$ denote the number of red/blue-colored paths of length $n$ with exactly $j$ blue vertices that do not contain a subpath of length $k$, then the probability of survival of a system where each component has identical, independent survival probability $p$ is 
\[
\displaystyle\sum_{j=0}^{k+1} F(n,k,j) p^j (1-p)^{n-j}.
\]

There are many useful and beautiful generalizations of this problem in both reliability theory and graph theory \cite{FA2005,  S1981, MKP2007, BP1999, P1990, VL2014}. 

However, many systems, both mechanical and biological, are better modeled by \textit{hypergraphs}, where we consider \textit{hyperedges} allowed to contain more than two vertices. Examples of such systems include complex molecular networks \cite{chemistry}, cancer formations processes \cite{OurPaper}, and modeling of the folksonomy \cite{physics}. Motivated by problems considered in reliability theory, as well as biology, we make the following definition, which generalizers the definition found in \cite{ZZYD2018}. 

\begin{definition}
Let $H$ and $K$ be hypergraphs, and let $p$ be a probability distribution. An \textit{$(H,K,p)$-reliable system} has the following attributes: \\
\begin{itemize}
    \item each vertex $v$ in $H$ has an independent identical distribution described by $p$, giving the probability that $v$ has failed.
    \item the system fails when there exists a collection of failed vertices that form a subhypergraph isomorphic to $K$.
\end{itemize}
\end{definition}

 The \textit{time to failure} of an $(H,K,p)$-reliable system is a random variable giving the time that it takes for a copy of $K$ consisting of failed vertices to exist inside $H$. The nature of the distribution for time to failure will be influenced by the distribution of time to failure for the individual vertices in the system, as well as by the overall shape of $H$: how many copies of $K$ it contains, as well as how they fit together inside $H$.

In effect, determining the probability of distribution of the time to failure for such a system is given by determining the number of ways in which exactly $j$ out of n failed vertices in the hypergraph could induce a failed $K$, along with the probability that $j$ vertices have failed. Since the latter probability calculation is simply $p^j(1-p)^{n-j}$, the problem of determining system lifetimes essentially reduces to a counting problem in hypergraphs, of the type we consider in this paper. 

This framework is extremely general and, to the best of our knowledge, has yet to be investigated outside of the previously mentioned cases when $H$ and $K$ are certain types of graphs. In this paper, we provide recurrences for these counts in the case when $H$ and $K$ are certain types of hypergraph paths and cycles. In addition to being a natural yet challenging starting point, these initial cases are motivated by the fact that they are easily detectable subhypergraphs that occur in larger hypergraphs. In practical applications, it is often enough to have upper or lower bounds on the probability of failure at a certain point. We make the following observations, which are not difficult to prove.   

\begin{itemize}
    \item If $G$ is a subhypergraph of $H$, then the time to failure of a $(G,K,p)$-reliable system will be less than or equal to that of an $(H,K,p)$ reliable system, since $G$ contains at least as many copies of $K$ as $H$ does.  
    \item If $G$ is a subhypergraph of $K$, then the time to failure of an $(H,G,p)$-reliable system will be less than or equal to that of an $(H,K,P)$ reliable system, since a failed copy of $K$ must exist before a failed copy of $K$ can. 
\end{itemize}

In addition to providing tools to address a probability calculation arising from reliability of certain systems defined using hypergraphs, we think these counts are of combinatorial interest in their own right. 

\section{Definitions}

We now provide some background and necessary definitions of hypergraphs, restricted to the case we consider. We denote the cardinality of set $A$ by $|A|$. We use $ \lfloor x \rfloor$ for the floor function applied to x, and $\lceil x \rceil$ for the ceiling function applied to $x$. The function of two parameters $_aR_b$ is the remainder function of $a$ divided by $b.$ We let $\binom{n}{k}$
represent the binomial coefficient $\displaystyle\frac{n!}{k!(n-k)!}$, with the convention that $\binom{n}{k} = 0$ if either $n$ or $k$ is negative.

\begin{definition}
An $r$-uniform hypergraph $H$ consists of a vertex set $V(H)$ and an edge set $E(H)$, where each $e \in E(H)$ is an $r$-element subset of $V(H)$. 
\end{definition}

\begin{definition}
An r-uniform hypergraph $P$ is a loose $r$-path if there is an ordering on the edges $E(P)$ where $|E_i\cap E_j|=1$ if $|i-j|=1$ and $|E_i\cap E_j|= 0$ otherwise. The length of a loose $r$-path $P$ is given $|E(P)|$. 
\end{definition}

\begin{definition}
Let $0<m<r$. An $r$-uniform hypergraph $P$ is a $m-$tight $r$-path with $n$ edges if there is an ordering on the vertices $0,1,...,h$ such that $E_1=\{0,1,2,...,r-1\}$, and $E_{i+1} = \{ v + m \mid v \in E_{i} \}$ for $0 \leq i \leq n-1$.   
\end{definition}

\begin{remark}
Note that an $r$-path is a $1-$tight $r$-path if and only if it is a loose $r$-path.
\end{remark}

\begin{definition}
Given a color $x\in \{red,blue\}$, we say that an edge is monochromatic color x, or simply that the edge is x if every vertex is colored x.
\end{definition}

\begin{definition}
A loose $r$-cycle of length $n$ is an $r$-uniform hypergraph which satisfies the following two conditions:
\begin{enumerate}
    \item Every edge has exactly $2$ vertices of degree $2$ and all other vertices of degree 1,
    \item the removal of any single edge, along with all the corresponding vertices of degree 1 from that edge, yields a loose $r$-path of length $n-1$.
\end{enumerate}
\end{definition}

\begin{definition}
A hypergraph is an $(r-1)$-tight $r$-cycle if it can be constructed, edge by edge, so that the first $n-r-2$ edges after the first edge each add exactly one vertex, the last $(r-1)$ edges add no vertices, for any $0<i<n$, $|E_i \cap E_{i+1}|=r-1$ and the removal of any $(r-1)$ consecutive edges (when read modulo n and retaining the vertices in each edge) yields an $(r-1)-$tight $r$-path. 
\end{definition}

\section{Results for loose paths and cycles}

\subsection{Results for loose paths}

In this subsection we derive recurrence relations for the number of ways to color the vertices of a loose $r$-path on $n$ edges while avoiding $k$ consecutive hyperedges with all vertices blue.

First we introduce notation for the quantities of interest. 
\begin{definition}
The number of ways we can color exactly $j$ vertices blue in an $r$-path with $n$ edges while avoiding a monochromatic blue subpath of length $k$ is denoted by $F_k(r,n,j)$
\end{definition}

To prove the main theorems of this section, we will need to use the following two easy proposition which we state without proof.

\begin{proposition}\label{PropBaseCases}
For $n\geq k,$ we have the following cases:
\begin{enumerate}
    \item If $j < k(r-1)+1$ then $ F_k(r,n,j)=\binom{n(r-1)+1}{j},$
       
    \item If $n\geq k+1, j > n(r-1)+1-\left\lceil\frac{n-(k-1)}{k+1}\right\rceil$ then $ F_k(r,n,j)=0.$
    \item if $n=k$, $j\geq n(r-1)+1$ then $ F_k(r,n,j)=0.$
\end{enumerate}

\end{proposition}


We will also need to introduce two auxiliary quantities of interest, related to the main one we consider. 

\begin{definition}
The number of ways we can color exactly $j$ vertices blue in an $r$-path with $n$ edges while avoiding a monochromatic blue subpath of length $k$ with the restriction that the first $i$ edges are blue and edge $(i+1)$ is not blue is denoted by $F_k^i(r,n,j)$.
\end{definition}

\begin{definition}
For a touple $(r,n,k,j)$ we say that $j$ is permissible if the touple is not described by any of the conditions of Proposition \ref{PropBaseCases}. 
\end{definition}

\begin{definition}
We use $F_k^*(h,n,j)$ to denote the number of ways to can color exactly $j$ vertices blue in an $r$-path with $n$ edges while avoiding a monochromatic blue subpath of length $k$ with the restriction that a fixed vertex in the first hyperedge  must be colored red.
\end{definition}

\begin{lemma}\label{l:F-star-relationship}
For positive integers $r,k$ and $m$, 
$$F_k^*(r,m,i)={\sum_{b=0}^{r-2}\binom{r-2}{b}F_k(r,m-1,i-b)}.$$
\end{lemma}

\begin{proof}
Since we know a fixed vertex of degree 1 in $E_1$ is red, then we have no worries that $E_1$ will be a blue edge. Thus we can pick any $0\leq b\leq r-2$ degree 1 vertices of $E_1$ to color blue. This leaves a hypergraph isomorphic to a path of length $m-1$ to color with exactly $i-b$ blue vertices, without introducing a blue hyperedge. Thus there are  $F_k(r,m-1,i-b)$ ways to color the remaining vertices.
\end{proof}

We can now begin proving the recurrence relations for loose paths. We begin with a few small but illuminating special cases. The first case is that of avoiding a single blue hyperedge.

\begin{theorem}\label{one-edge}
Let $r\geq 3, n>1$ and $j$ be permissible,  
\begin{align*} F_1(r,n,j) &= 
{\sum_{i=0}^{r-2}}  \binom{r-1}{i} F_1(r,n-1,j-i) \\
& +{\sum_{i=0}^{r-2}} \binom{r-2}{i} F_1(r,n-2,j-(r-1)-i).
\end{align*}
\end{theorem}

\begin{proof}
Number the vertices $1,2,3,...,r$ $\in$ $E_1;$ $r,r+1,r+2,r+3,...,2r-1\in E_2$, etc. and consider the degree-1 vertices $1, 2, \ldots, r-1$ in the leaf $E_1$. We have the following cases:
\begin{enumerate}
    \item If none of these vertices are blue, then we have exactly $F_1(r,n-1,j)$ ways to color theremaining vertices in the remaining $n-1$ edges.
    \item If we assume that $0<k\leq r-2$ of these vertices are colored blue, then we have $\binom{r-1}{k}$ ways to choose which of these vertices to color blue. We then have $F_1(r,n-1,j-k)$ ways to color the remaining $n-1$ edges of the hypergraph. In total, this gives $\binom{r-1}{k} F_1(r,n-1,j-k)$ distinct colorings for this case. 
    \item If all $(r-1)$ of these vertices are colored blue, then vertex $r$ must be colored red. Then we have exactly $F_1^*(r,n-1,j-(r-1))$ possibilities for coloring the remainder of the hypergraph within the given constraints.
\end{enumerate} 

\noindent From these cases we see that 

$$F_1(r,n,j)= F_1(r,n-1,j)+ {\sum_{i=1}^{r-2}\binom{r-1}{i}F_1(r,n-1,j-i)}+ F_1^*(r,n-1,j-(r-1)).$$

\noindent Applying Lemma~\ref{l:F-star-relationship} follows that

$$F_1(r,n,j)=  {\sum_{i=0}^{r-2}\binom{r-1}{i}F_1(r,n-1,j-i)}+ \sum_{i=0}^{r-2}\binom{r-2}{i}F_1(r,n-2,j-(r-1)-i),$$

\noindent proving the theorem.

\end{proof}

The following theorem can be easily deduced from the results in \cite{S1981}, motivated by reliability theory; for completeness, we state and prove it now. 

\begin{theorem}\label{2thm}
For $k\geq 1,$ and permissible $j$
\begin{align*}
F_k(2,n,j) &= F_k(2,n-2,j-1)+F_k(2,n-1,j)  \\
& + \sum_{i=1}^{k-1}F_k(2,n-i-2,j-i-1).
\end{align*}
\end{theorem}

\begin{proof}
Consider the first non-monochromatic blue edge in our 2-Path. If that is the first edge then we have two possibilities. If the first vertex is blue, we have $F_k(2,n-2,j-1)$ ways to color the remaining vertices. Otherwise, the first vertex is red, giving $F_k(2,n-1,j)$ ways to color the remaining vertices. Alternatively, let us assume that the first $1<i<k$ edges are colored monochromatically blue, but the $(i+1)^{st}$ edge is not. Then vertex $V_{i+1}$ is blue and vertex $V_{i+2}$ is red, yielding $F_k(2,n-i-2,j-i-1)$  ways to color the remaining vertices of the graph.
\end{proof}

The remaining results of this paper generalize the formula from Theorem \ref{2thm} to $r-$uniform hypergraphs for $r>2.$ This is done via Theorem \ref{k-path-3-uniform} which gives us the corresponding recurrence for $3-$uniform hypergraphs and Theorem \ref{r-uniform-k-paths} which provides the general case for $r>3.$

\begin{theorem}\label{k-path-3-uniform}
If $k>1,n>k$ and $j$ be permissible, then we have 

\begin{align*} F_k(3,n,j)&= F_k(3,n-1,j)+2F_k(3,n-1,j-1) \\ &+ F_k(3,n-2,j-2) 
+ F_k(3,n-2,j-3) \\ & +\sum_{i=1}^{k-1}F_k(3,n-(l+1),j-2l-1) \\
& +\sum_{i=1}^{k-1}F_k(3,n-(l+2),j-2(l+1)) \\ & +\sum_{i=1}^{k-1}F_k(3,n-(l+2),j-2(l+1)-1).
\end{align*}
\end{theorem}

\begin{proof}
Note that $F_k(3,n,j)=\displaystyle{\sum_{i=0}^{k-1}}F^i_k(3,n,j).$ To calculate $F_k^0(3,n,j)$, we assume that  the first edge is not monochromatic blue, and consider the vertices in that edge. One option is to select a single vertex of degree 1 to color blue, giving $2F_k(3,n-1,j-1)$ ways to color our hyperpath. Alternatively,  we could pick none of the vertices of degree 1 in $E_1$ to color blue, in which case we have $F_k(3,n-1,j)$ ways to color our hyperpath. Finally, we could pick both degree 1 vertices from $E_1$ to color blue in which case, since the edge is not monochromatic blue, the vertex of degree 2 in $E_1$ must be colored red. Thus let us consider the vertices in $E_2.$ Since $V_3$ is colored red then we can either choose to color the unique degree 1 vertex of $E_2$ blue or not leaving us a total of $F_k(3,n-2,j-2)+F_k(3,n-2,j-3)$ ways to color the remaining hyperedges. From these observations we see that:

\begin{align*}
F_k^0(3,n,j) &=2F_k(3,n-1,j-1)+F_k(3,n-1,j) \\
& +F_k(3,n-2,j-2)+F_k(3,n-2,j-3).
\end{align*}

We now turn our attention to $F^i_k(3,n,j)$ for $i>0$. Assume that the first $i$ edges are monochromatic blue while the $(i+1)$ edge is not. The vertices of this edge are $V_{i(2)+1},V_{i(2)+2},V_{i(2)+3},$ where $E_l\cap E_{i+1}=V_{i(2)+1}, E_{i+1}\cap E_{i+2}=V_{i(2)+3}.$ We note then that $V_{i(2)+1}$ must be colored blue. If $V_{i(2)+2}$ is colored red then we have a total of $F_k(3,n-(i+1),j-2i-1)$ ways to color the remaining edges, otherwise if $V_{i(2)+2}$ is colored blue then $V_{i(2)+3}$ must be colored red. In this case we will consider the color of the degree 1 vertex $V_{i(2)+4}.$ If this vertex is colored red then we have $F_k(3,n-(i+2),j-2(i+1))$ ways to color the remaining edges. Otherwise if the vertex is colored blue then we have $F_k(3,n-(i+2),j-2(i+1)-1)$ to color the remaining edges. We see then that:

\begin{align*}
F^i_k(3,n,j) &= F_k(3,n-(i+1),j-2i-1)+F_k(3,n-(i+2),j-2(i+1))\\
&+ F_k(3,n-(i+2),j-2(i+1)-1). 
\end{align*}

\end{proof}

Having walked through the proof of Theorem \ref{k-path-3-uniform}, we will now see that the general case is very similar.  

\begin{theorem}\label{r-uniform-k-paths}
Let $k>1,r>3,n>k$, $M=(n(r-1)+1)$, and $N=((k-1)(r-1)+2(r-2)+1)$, with $j$ permissible

$$F_k(r,n,j) =\sum_{i=0}^{r-2}\binom{r-2}{i}F_k(r,n-2,j-i-(r-1))$$
$$+\sum_{i=0}^{r-2}\binom{r-1}{i} F_k(r,n-1,j-i)$$
$$+\sum_{l=1}^{k-1}\sum_{i=0}^{r-3}\binom{r-2}{i}F_k(r,n-(i+1),j-l(r-1)-1-i)$$
$$+\sum_{l=1}^{k-1}\sum_{i=0}^{r-2}\binom{r-2}{i}F_k(r,n-(i+2), j-(i+1)(r-1)-i)$$
\end{theorem}

\begin{proof}
Note that $F_k(r,n,j)=\sum_{i=0}^{k-1}F_k^i(r,n,j)=F_k^0(r,n,j)+\sum_{l=1}^{k-1}F_k^l(r,n,j).$  

\noindent Note that $F_k^0(r,n,j)$ is of concern if and only if the first edge in the hypergraph is not monochromatic blue. In this edge we can pick $i\leq r-2$ vertices of degree 1 from edge 1 to color blue. If this happens we have $j-i$ vertices remaining to color red in $n-1$ edges with no additional restrictions so we have  $F_k(r,n-1,j-i)$ ways to color the remaining edges of the graph. We also have $\binom{r-1}{i}$ ways to pick the vertices of degree 1. From these observations we find that

$$F_k^0(r,n,j)=x+\sum_{i=0}^{r-2}\binom{r-1}{i} F_k(r,n-1,j-i),$$

where $x$ is the number of ways to color the remaining hyperedges of the path given that all the vertices of degree 1 are colored blue. 

To explore $x$ consider the path $\Gamma$ consisting of edges $E_2,E_3,...,E_n$. Relabel the vertices in this graph $V_1,V_2,V_3...$ such that $V_1\in E_1\cap E_2$. Then $V_1$ must be colored  red. Thus $E_2$ contains $r-2$ vertices of degree 1 that may be colored either blue or red. Let us pick $i$ of these and color them blue, then we have $F_k(r,n-2,j-i-(r-1))$ ways to color the remaining edges. We have exactly $\binom{r-2}{i}$ ways to pick these $i$ vertices thus we see that $x=\sum_{i=0}^{r-2}\binom{r-2}{i}F_k(r,n-2,j-i-(r-1))$.

From this last statement we find that

$$F_k^0(r,n,j)=\sum_{i=0}^{r-2}\binom{r-2}{i}F_k(r,n-2,j-i-(r-1))+\sum_{i=0}^{r-2}\binom{r-1}{i} F_k(r,n-1,j-i).$$


Now let $l>0,$ and assume that edges $E_1,E_2,...,E_l$ are monochromatically colored blue and that $E_{l+1}$ is not a monochromatic blue edge. Then we see we have used $l(r-1)+1$ blue vertices in the first $l$ edges and so we are left with $j-l(r-1)-1$ blue vertices to place in the graph formed by edges $E_{l+1},E_{l+2},...,E_n.$ Let us rename the vertices in $E_{l+1}$ as $V_1,V_2,...,V_r$ such that $V_1\in E_l \cap E_{l+1}.$ Then $V_1$ must be colored blue. We may choose $m\leq r-3$ vertices to color blue and so we see that we have $F_k(r,n-l-1,j-l(r-1)-1-i)$ ways to color the remaining edges.

If, however, we color all $r-2$ vertices of degree 1 in $E_{l+1}$ blue, then vertex $V_r$ must be colored red (since otherwise $E_{l+1}$ would be monochromatic blue). We have therefor used $(l+1)(r-1)$ vertices blue and so we have $j-(l+1)(r-1)$ left to place in the remaining $n-l-1$ edges. Let us then consider how we might color the edges $E_{l+2},E_{l+3},...,E_n$  Once again, let us rename vertex $V_1\in E_{l+1}\cap E_{l+2}.$ Then $V_1$ is red thus we may pick $m\leq r-2$ of the vertices of degree 1 to color blue with no fear of the edge becoming monochromatic blue. We will have have exactly $F_k(r,n-(l+1), j-(l+1)(r-1)-m)$ ways to color the remaining hyperedges. For each permissible value of $m$ we have $\binom{r-2}{m}$ ways to pick the vertices of degree 1 in edge $E_{l+2}.$ Thus we see that 

$$F_k^l(r,n,j)=\sum_{i=0}^{r-3}\binom{r-2}{i}F_k(r,n-(l+1),j-l(r-1)-1-i)$$
$$+\sum_{i=0}^{r-2}\binom{r-2}{i}F_k(r,n-(l+2), j-l(r-1)-r-3-i),$$

\noindent thus proving the desired result.
\end{proof}

We will re-state a fact just proven in Theorem~\ref{r-uniform-k-paths} for future use. 

\begin{proposition}
Let $k>1,r>3,n>k$, $M=(n(r-1)+1)$, and $N=((k-1)(r-1)+2(r-2)+1).$ Assume 
$j$ is permissible then

$$F_k^l(r,n,j)=\sum_{i=0}^{r-3}\binom{r-2}{i}F_k(r,n-l,j-l(r-1)-1-i)$$
$$+\sum_{i=0}^{r-2}\binom{r-2}{i}F_k(r,n-(l+1), j-(l+1)(r-1)-i).$$
\end{proposition}

\section{Applications }

We begin this section with two theorems that tie together colorings for tight paths and cycles to colorings of 2-paths and 2-cycles (i.e. the usual case of graphs). 

\begin{definition}
The number of ways we can color exactly $j$ vertices blue in a $(r-1)$-tight $r$-path with $n$ edges while avoiding a monochromatic blue subpath of length $k$ is denoted by $T_k(r,n,j).$
\end{definition}

\begin{theorem}
Let $k>r,$ then 
$$T_k(r,n,j)=F_{r+k-2}(2,r+n-2,j)$$
\end{theorem}

\begin{proof}
Note that any $r+k-1$ blue successive vertices create a subpath of length k.  
\end{proof}

We now define several quantities related to colorings in hypergraph cycles. 

\begin{definition}
The number of ways we can color exactly $j$ vertices blue in a loose $r$-cycle with $n$ edges while avoiding a monochromatic blue subpath of length $k$ is denoted by $C_k(r,n,j)$
\end{definition}

\begin{definition}
The number of ways we can color exactly $j$ vertices blue in a loose $r$-cycle with $n$ edges while avoiding a monochromatic blue subpath of length $k$ with the extra restriction that the edge $E_n$  is  blue is denoted by $C_k^{b}(r,n,j)$
\end{definition}

\begin{definition}
The number of ways to color an $(r-1)-$tight $r$-cycle on $n$ edges with exactly $j$ blue vertices while avoiding a blue monochromatic subpath of length $k$ is $TC_k(r,n,j)$. 
\end{definition}

\begin{theorem}
For any integers $k,r,j > 2$, $$TC_k(r,n,j)=C_k(2,r+n-2,j).$$
\end{theorem}

\begin{proof}
Note that any consecutive string of $k$ vertices produces an $(r-1)-$tight path. 
\end{proof}

We now turn our attention to the problem of determining coloring numbers for loose cycles.

\begin{definition}
The number of ways we can color exactly $j$ vertices blue in a loose $r$-cycle with $n$ edges while avoiding a monochromatic blue subpath of length $k$ with the extra restriction that the edge $E_n$  is not blue is denoted by $C_k^{nb}(r,n,j)$
\end{definition}

\begin{theorem}
For positive integers $n$ and $j$ with $n \geq 4$, $$C_k(2,n,j)=F_k(2,n-2,j)+\sum_{l=0}^{k-1}(l+1)F_k(2,n-l-4,j-l-1).$$
\end{theorem}

\begin{proof}
Assume that we have a permissible coloring on a cycle of size $n$. Then we have two possibilities:
\begin{enumerate}
    \item The vertex $V_n$ is red, in which case its removal gives us a permissible path coloring; on the other hand, given a permissible coloring of a path of length $(n-2),$ we can add a red vertex $V_n$ to create a permissible coloring on a cycle.
    \item If $V_n$ is blue, and lies on a maximal monochromatic subpath of length $l\geq 0.$ Then  the removal of all the vertices in this subpath leaves us with a permissible coloring on $n-l-1$ vertices, where both the first at last vertex are red. Thus there are $F_k(r,n-l-4,j-l-1)$ ways to color the remaining vertices. Note that we assumed that $V_n$ was part of a vertex on a monochromatic subpath of length $l,$ however, $V_n$ could have been in any of the $l+1$ positions for a vertex on that subpath.
\end{enumerate}
\end{proof}

Through the same argument we see that for any $r$ the following holds.

\begin{theorem}\label{cycles}
Let $n>k+2$, $r>2$ \ \\  $M=(n(r-1)+1)$, and $N=((k-1)(r-1)+2(r-2)+1)$,

 with $j<[M/N](N-1)+_NR_M.$  Then
$$C_k(r,n,j)=C^{nb}_k(r,n,j)+C^b_k(r,n,j)$$

\end{theorem}

\noindent Of course Theorem \ref{cycles} is of no use without formulas for both $C_k^b(r,n,j)$ and $C_k^{nb}(r,n,j),$ which we now provide via Theorem \ref{cycles1}, and Theorem \ref{cycles2}.

\begin{theorem}\label{cycles1}
Let $n>k+2$, $r>2$ \ \\  $M=(n(r-1)+1)$, and $N=((k-1)(r-1)+2(r-2)+1)$, with

 $j<[M/N](N-1)+_NR_M.$  Then

 $$C^{nb}_k(r,n,j)=\sum_{i=0}^{r-3}\binom{r-2}{i}F_k(r,n-1,j-i)$$

  $$+\sum_{i=0}^{2r-4}\binom{2r-4}{i}F_k(r,n-3,j-i-(r-2))$$
 
 $$+2\sum_{l=1}^{k-1}\sum_{i=0}^{r-2}\binom{r-2}{i}F^l(r,n-2,j-(r-2)-i)$$
 
 $$+2\sum_{b=0}^{r-3}\sum_{a=0}^{r-2}\binom{r-2}{a}\binom{r-2}{b}F(r,n-3,j-(r-2)-1-(a+b))$$
 
$$+2\sum_{i=0}^{2r-4}\binom{2r-4}{i}F_k(r,n-4,j-2(r-2)-1-i). $$

\end{theorem}

\begin{proof}

Assume that we have a loose cycle of length $n$ such that there is no monochromatic blue subpath of length $k$. Then let $E_n$ be an edge that is not blue.  If $E_n$ has $i<r-2$ blue vertices of degree 1 then we have $F_k(r,n-1,j-i)$ ways to color the remaining edges.   Otherwise, if all $r-2$ vertices are colored blue then we have one of two cases to consider:
\begin{enumerate}
    \item \textbf{Both degree 2 vertices of $E$ are red.} For convenience, let us denote by $E_1$ and $E_{n-1}$ the two edges with noneempty intersection with $E_n$. Then we can pick any combination of $i\leq 2r-4$ degree 1 vertices from $E_1$ and $E_{n-1}$ to color blue and we are left with $F_k(r,n-3,j-(r-2)-i)$ ways to color the remaining vertices.
    \item \textbf{Exactly one of the 2 vertices of degree 2 are blue.} Without loss of generality we use $E_1$ to refer to the edge overlapping with $E_n$ which contains the blue vertex, we use $E_{n-1}$ to refer to the edge overlapping with $E_n$ which contains the red vertex. Furthermore, let us assume that $a\leq r-2$ vertices  of degree 1 from $E_{n-1}$ are colored blue. Then we have the following subcases:
    \begin{enumerate}
        \item $E_1$ is blue, then $E_1$ is part of a blue subpath of length $l\geq 1,$ thus we have $F_k^l(r,n-2,j-(r-2)-a)$ ways to color the remaining vertices.
        \item $E_1$ is not blue, then we are left with two last subcases.
        \begin{itemize}
            \item At least one degree 1 vertex from $E_1$ is red, then if we have $i$ such red vertices we find that we have $F_k(r,n-3,j-(r-2)-1-(i+a))$ ways to color the vertices of the remaining edges.
            \item All degree 1 vertices from $E_1$ are blue which means that the vertex $v\in E_1\cap E_2$ must be red. Thus we can pick  $0\leq b\leq r-2$ vertices of degree 1 from $E_2$ to color blue leaving us with $F_k(r,n-4,j-2(r-2)-1-(a+b))$ ways to color the vertices of the remaining edges. 
        \end{itemize}
    \end{enumerate}
 \end{enumerate}
 
 Of course, we could reverse the assumption that $v_1$ is red and $v_{(n-1)(r-1)+1}$ is blue thus all quantities from case $(2)$ must be doubled.
\end{proof}

\begin{theorem}\label{cycles2}
Let $n>k+2$, $r>2$ \ \\  $M=(n(r-1)+1)$, and $N=((k-1)(r-1)+2(r-2)+1)$ with

 $j<[M/N](N-1)+_NR_M.$  \newpage
 
 Then

 $$C^b_k(r,n,j)$$
$$=\sum_{l=1}^{k-1}\sum_{a=0}^{r-3}\sum_{b=0}^{r-2}l\binom{r-2}{a}\binom{r-2}{b}F_k(r,k-l-2,j-l(r-1)-1-a-b)$$

$$+2\sum_{l=1}^{k-1}\sum_{a=0}^{r-3}\sum_{b=0}^{r-2}l \binom{r-2}{a}\binom{r-2}{b}F_k(r,n-l-3,j-l(r-1)-1-a-b-(r-2))$$

$$+\sum_{l=1}^{k-1}\sum_{a=o}^{2r-4}l \binom{2r-4}{i}F(r,n-l-4,j-l(r-1)-1-2(r-2)-a)$$
\end{theorem}


\begin{proof}

Let us assume that $E_n$ is blue, then let $k>l\geq 1$ be the length of the maximum blue subpath containing $E_n$. Let us call this subpath $\Gamma.$ Then $E(\Gamma)$ can be partitioned into three disjoint sets $E_{(<k)}(\Gamma)=E(\Gamma) \cap \{E_1,E_2,...,E_{k-1}\}$ and $E_{(>k)}(\Gamma) \cap \{E_{k+1}, E_{k+2},...,E_{n-1}\},$ and $E_{\omega}={E_n}$ then $|E_{(<k)}(\Gamma)|+|E_{(>k)}(\Gamma)|=l.$

Let us define $M,m \in \{0,1,2,3,...,n-1\}$ such that $M$ is the largest index $i$ where $E_i\in E_{<k}(\Gamma)$ or $M=0$ if $E_{<k}(\Gamma)$ is empty,  and $m$ is the smallest index $u$ such that $E_u \in E_{>k}(\Gamma),$ or $m=0$ if  $E_{>k}(\Gamma)$ is empty. 

Note therefore that edges $E_{M+1}$ and $E_{m-1}$ are not blue but each have at least one degree two vertex that is blue. Based on this assumption we have the following cases:

Let $i_{M+1}$ be the number of degree 1 vertices from $E_{M+1}$ that are blue and $i_{m-1}$ be the number of blue degree 1 vertices from $E_{m-1}$. There are four cases:
\begin{enumerate}
\item \textbf{Both $i_{M+1},i_{m-1}<r-2:$} then there are $F_k(r,n-l-2,j-l(r-1)-1-i_{M+1}-i_{m-1})$ of coloring the remaining vertices. 

\item \textbf{$i_{M+1} <r-2$ and $i_{m-1}=r-2$ :}
In this case, the remaining degree 2 vertex of edge $E_{m-1}$ must be red then let $i_{m-2}\leq r-2$ be the number of blue degree 1 vertices from $E_{m-2},$ then we have $F_k(r,n-l-3,j-l(r-1)-1-i_{M+1}-(r-2)-i_{m-2})$.

\item \textbf{$i_{M+1} =r-2$ and $i_{m-1}<r-2$ :} By a similar argument to the previous case we have $F_k(r,n-l-3,j-l(r-1)-1-i_{m-1}-(r-2)-i_{M+2})$ ways of coloring the vertices where $i_{M+2}\leq r-2$ are the number of degree 1 vertices in $E_{M+2}$. 

\item \textbf{Both $i_{M+1} =r-2$ and $i_{m-1}=r-2$ :}

Then the remaining degree 2 vertices of both $E_{M+1}$ and $E_{m-1}$ are red then let $i_{M+2}$ and $i_{m-2}$ be the number of degree 1 blue vertices from $E_{M+2}$ and $E_{m-2}$ respectively. We see then that we have $F_k(r,m-l-4,n-l(r-1)-1-2(r-2)-i_{m-2}-i_{M+2})$.
\end{enumerate}

\noindent Thus completing the proof of the theorem.
\end{proof}

\section{Conclusion and future directions}
In this paper we provided a recursive formula for the number of ways to place $j$ blue vertices in a loose r-uniform hyperpath on length $n$ while avoiding a monochromatic blue subhypergraph of length $k$. We then used these results to solve the corresponding problem for both $(r-1)$-tight r-paths and $(r-1)$-tight r-cycles, as wellas  loose r-cycles. There are several natural questions and conjectures raised by this investigation, these include:
\begin{enumerate}
    \item How many ways are there to place $j$ blue vertices in an m-tight r-uniform hyperpath of length $n$ while avoiding a monochromatic blue subhypergraph of length $k$?
    \item How many ways are there to place $j$ blue vertices in an m-tight r-uniform hypercycles of length $n$ while avoiding a monochromatic blue subhypergraph of length $k$?
    \item Can we find a close approximation to $F_k(r,n,j)$ which would be more efficient than computing the number directly? Such approximations are used widely in engineering for $F_k(2,n,j)$ \cite{ VL2014}.
\end{enumerate}

We offer the following conjecture that may be useful in obtaining helpful approximations. 

\begin{conjecture}
Let $1 \leq m\leq r-1,$ $F_k(m,r,n,j)$ be the number of ways to place $j$ blue vertices in a path of size $n$ in an $m-$tight $r-$path. Then
$$ F_k(r-1,r,n,j)\leq F_k(m,r,n,j)\leq F_k(1,r,n,j). $$
\end{conjecture}

\bibliographystyle{amsplain}
\bibliography{number-of-placements.bib}

\end{document}